\documentclass[pagesize]{scrartcl}
\usepackage{graphicx}
\usepackage{enumerate}
\usepackage{amssymb, amsmath, amsthm}
\usepackage[round]{natbib}
\usepackage{pgfplots}

\usepackage{hyperref}
\usepackage{tikz}
\usetikzlibrary{arrows,automata}

\newtheorem{proposition}{Proposition}[section] %denna med i latex
\newtheorem{defn}[proposition]{Definition}

\newtheorem{thm}[proposition]{Theorem}
\newtheorem{lemma}[proposition]{Lemma}

\theoremstyle{remark}
\newtheorem{remark}[proposition]{Remark}

\usepackage[T1]{fontenc}
\usepackage{lmodern}

\newcommand{\A}{\ensuremath{{\mathcal A}}}
\newcommand{\N}{\ensuremath{{\mathbb N}}}
\newcommand{\C}{\ensuremath{{\mathcal C}}}
\newcommand{\R}{\ensuremath{{\mathbb R}}}
\newcommand{\Z}{\ensuremath{{\mathbb Z}}}

\begin{document}
\title{A General Method for Finding the Optimal Threshold in Discrete Time}
%\title{On explicit solutions of optimal stopping problems for strong Markov processes}
\author{S\"oren Christensen\thanks{Universit\"at Hamburg, Department of Mathematics, Research Group Statistics and Stochastic Processes, Bundesstr. 55 (Geomatikum), 20146 Hamburg, Germany} 
, Albrecht Irle\thanks{Christian-Albrechts-Universit\"at, Mathematisches Seminar, Ludewig-Meyn-Str. 4, 24098 Kiel, Germany}}
\date{\today}
\maketitle

\begin{abstract}
We develop an approach for solving one-sided optimal stopping problems in discrete time for general underlying Markov processes on the real line. The main idea is to transform the problem into an auxiliary problem for the ladder height variables. In case that the original problem has a one-sided solution and the auxiliary problem has a monotone structure, the corresponding myopic stopping time is optimal for the original problem as well. This elementary line of argument directly leads to a characterization of the optimal boundary
in the original problem: The optimal threshold is given by the threshold of the myopic stopping time in the auxiliary problem. Supplying also a sufficient condition for our approach to work, we obtain solutions for many prominent examples in the literature, among others the problems of Novikov-Shiryaev, Shepp-Shiryaev, and the American put in option pricing under general conditions. As a further application we show that for underlying random walks (and L\'evy processes in continuous time), general monotone and log-concave reward functions $g$ lead to one-sided stopping problems. 
%. We furthermore give a sufficient condition for the approach to work. As an application, we show  for general underlying random walks that the reward functions $g$ leading to one-sided stopping problems are exactly the monotone and log-concave functions. This generalises many existing results in the literature. 
%The method is illustrated with a variety of examples.
\end{abstract}


\vspace{.2cm}

\section{The Basic Setup}\label{sec:setup}
We look at the optimal stopping problem for
\[\rho^ng(Y_n),\;n=0,1,2,\dots,\;0<\rho\leq 1,\;g:\R\rightarrow[0,\infty).\]
Here, $(Y_n)_{n\in\N_0}$ is a Markov process with respect to some filtration $(\A_n)_{n\in\N_0}$, $\A_\infty:=\sigma(\bigcup_{n\in\N_0}\A_n)$, having starting point $y$ with respect to $P_y=P(\cdot|Y_0=y)$. The state space is some subset of $\R$. 

In certain examples of interest, stopping times may take the value $+\infty$ with positive probability, e.g., the optimal stopping time in the well-known Novikov-Shiryaev problem as discussed below. For easier notation, we formally introduce an additional state $\Delta$ (not in $\R$) and set\ $Y_\infty=\Delta,\;g(\Delta)=0,\;\rho^\infty=0.$

We then define 
\begin{align}\label{eq:OSP}
V(y)=\sup_\tau E_y\rho^\tau g(Y_\tau)=\sup_\tau E_y\rho^\tau g(Y_\tau)1_{\{\tau<\infty\}},
\end{align}
where the supremum is taken with respect to all stopping times $\tau$. Define the stopping set
\[S^*=\{y:V(y)=g(y)\}.\]
Then the well-known candidate for an optimal stopping time is
\[\tau^*=\inf\{n:Y_n\in S^*\}\]
with $\inf\emptyset=\infty.$ In our situation ($g\geq0$), $\tau^*$ is optimal if, with respect to every $P_y$, the following condition holds:
\begin{align}\tag{Opt}\label{eq:Opt}
E_y\sup_n\rho^ng(Y_n)<\infty\mbox{ and }\rho^ng(Y_n)\rightarrow0\mbox{ as }n\rightarrow\infty,
\end{align}
see \cite{S}, 2.5. This condition will be fulfilled in our examples under well-known assumptions. Starting from the general theory, it then becomes the main task to determine the set\ $S^*$ as explicitly as possible. 

In particular for underlying random walks (and, resp., L\'evy processes in continuous time), this question has been studied in detail. We just want to mention \cite{DLT} for an early treatment and, more recently, \cite{Mo} for $g(x)=(K-{e}^x)^+$ and \cite{NS,NS2} for $g(x)={(x^+)}^\nu,\;\nu>0$, moreover \cite{AK}, \cite{KyS}, \cite{mishura2013optimal}, and \cite{mordecki2015optimal} for other reward functions. On the other hand, the idea of the particular solutions was generalized to larger classes of reward functions $g$ in \cite{BL2}, \cite{DLU}, \cite{Su}, \cite{MoSa}, \cite{Sa}, \cite{hsiau2014logconcave}, and \cite{boguslavskaya2014method}. All results have in common that the optimal stopping set $S^*$ is a one-sided interval. References for problems with other underlying one-dimensional Markov processes are given in\ Section \ref{sec:ex} below. 

The aim of this paper is to study the problem \eqref{eq:OSP} in a general discrete-time framework and to give an elementary, unifying approach for the treatment of these problems.

\section{An Elementary Approach}
The following approach, called an \emph{elementary approach}, was formulated in \cite{CIN} for the discrete time case for underlying autoregressive processes of order 1 and was taken up by \cite{baurdoux2013direct} for L\'evy processes:

In the first step we show by elementary arguments that $S^*$ is a one-sided interval $[a,\infty)$ or $(a,\infty)$ (resp. $(-\infty,a]$ or $(-\infty,a)$ in the reversed situation, see Remark \ref{rem:reversed}). Let us denote the optimal level by $a^*$. As these elementary arguments alone usually do not show whether $S^*$ is closed or open, we use the short-hand notation $I(a^*)$ for the optimal one-sided interval, so that
\[S^*=I(a^*)\mbox{ with $I(a^*)=[a^*,\infty)$ or $(a^*,\infty)$}.\]
The optimal stopping time then may be written as
\[\tau^*=\inf\{n\geq 0:Y_n\in I(a^*)\}. \]

%
%In a first step show by elementary arguments that $S^*$ is an intervall of the form 
%\[S^*=[a,\infty)\mbox{ for some }a\in\R\]
%(or $S^*=(-\infty,a]$, see Remark \ref{rem:reversed}). In the second step, find the optimal level $a$, which we denote by $a^*$. 
Elementary, of course not being a well-established mathematical term, here means that only a few lines of arguments, combined with, e.g., some common inequalities, are needed to provide the first step. 
It is interesting to note that this is possible for many prominent and new examples. This will be thoroughly treated in our paper and we present a new general result in this respect, see Theorem \ref{thm:optimal} and Appendix \ref{sec:appendix_illust}. As a consequence, we obtain a large class of reward functions $g$ leading to one-sided stopping problems for general underlying random walks: monotone, log-concave functions, see Theorem \ref{prop:RW_optimal}. 

The second step then is to find -- as explicitly as possible -- the threshold $a^*$ such that
\[\tau^*=\inf\{n\geq 0:Y_n\in I(a^*)\}\]
is optimal. 
The second main contribution of this paper is a general method to carry this out. This method is also elementary and gives a unifying way to solve discrete time stopping problems with a one-sided structure. We present this result in Theorem \ref{thm:opt_threshold}. The main idea is to transform the problem into an auxiliary for the ladder height variables. If this auxiliary is a monotone case stopping problem, the corresponding myopic stopping time is optimal for the original problem as well, and the optimal threshold is given by the threshold of the myopic stopping time in the auxiliary problem which has an explicit expression in terms of a function $\phi$ introduced in Definition \ref{defi:phi} below.

We treat a variety of new and well-known examples to illustrate our method in Section \ref{sec:ex}, among others the Novikov-Shiryaev problem, the Russian option and the American put problem.
We also discuss the relation to other approaches in the literature after having established our results, see Subsection \ref{subsec:other_approaches}. 
Although this paper focusses on the discrete time setting, a generalization to continuous time problems is of course also possible. This is discussed in Section \ref{sec:cont_time}.

%
%The references to the latter examples will illustrate this. The second step is then, in principle, a one-dimensional optimization problem of finding a maximum point of function
%\[a\mapsto E_y\rho^{\tau_a} g(Y_{\tau_a}),\]
%where\ $\tau_a:=\inf\{n:Y_n\geq a\}$ is a threshold time (and $y$ a suitable initial state). However, the expectation\ $E_y\rho^{\tau_a} g(Y_{\tau_a})$ cannot be found explicitly in most situations of interest and even numerical evaluations are often not straightforward. 

\section{The ladder variable reduction}\label{sec:sec_step}
Let us consider a stopping problem as in Section \ref{sec:setup} with
\[\rho^ng(Y_n),\;n=0,1,2,\dots\]
such that
\[\{y:g(y)>0\}=\{y:y>b\}\mbox{ for some }b\in\R\cup\{-\infty\}.\]
Assume that we have successfully performed the first step, so that
\[S^*=I(a^*),\;\tau^*=\inf\{n\geq0:\,Y_n\in I(a^*)\}.\]
Except for trivial situations we do not stop when $g(y)=0$ and have $V(b)>0$. This implies $a^*>b$ and will be assumed for our discussion. Due to the structure of $S^*=I(a^*)$, with $a^*>b$, we look at ascending ladder variables. We define
\begin{align*}
\sigma_0&=\inf\{j:Y_j>b\},\;\;\;\sigma_0=0\mbox{ if }b=-\infty,\\
\sigma_n&=\inf\{j>\sigma_{n-1}:Y_j>Y_{\sigma_{n-1}}\},\;n\geq 1,\\
\sigma_\infty&=\infty.
\end{align*}
$\sigma_n$ may take the value $+\infty$, and $\sigma_n=\infty$ implies $\sigma_m=\infty$ for $m>n$.  

The following easy result is fundamental for our further considerations. It shows that to find the optimal level\ $a^*$ we only have to solve an auxiliary optimal stopping problem for the ascending ladder process. It should, however, been noted that the distribution of the ladder process is seldom known explicitly. Exceptions are skip-free processes and processes with upward jumps of an exponential- or, more generally, a phase-type-structure, see \cite{asmussen1992phase} for the random walk case and the discussion in Subsection \ref{subsec:exponential_innov} for autoregressive processes. For more general processes, approximations have to be used. The examples in Section \ref{sec:ex} below, however, illustrate that often not the full distribution is needed, but just certain moments. 

\begin{proposition}
	Let the first entrance time into $S^*=I(a^*)$ be optimal and write $\mu^*=\inf\{m\geq 0:Y_{\sigma_m}\in I(a^*)\}$. Then $\tau^*=\sigma_{\mu^*}$ and $\mu^*$ is optimal with regard to the optimal stopping problem for
	\[\rho^{\sigma_m}g(Y_{\sigma_m}),\;m\in\N_0,\]
	with respect to the filtration $(\C_m)_{m\in\N_0},\;\C_m=\A_{\sigma_m}$.
\end{proposition}
\begin{proof}
Entering $I(a^*)$ can only occur at one of the $\sigma_m$, hence
\[\tau^*\in\{\sigma_m:0\leq m\leq \infty\}.\]
On $\{\tau^*<\infty\}\cap\{\tau^*=\sigma_m\}$ it holds that
\[\rho^{\tau^*}g(Y_{\tau^*})=\rho^{\sigma_m}g(Y_{\sigma_m}),\]
on $\{\tau^*=\infty\}$
\[\rho^{\tau^*}g(Y_{\tau^*})=0=\rho^{\sigma_\infty}g(Y_{\sigma_\infty}).\]
So,
\[\tau^*=\sigma_{\mu^*}\mbox{ with }\mu^*=\inf\{m\geq 0:Y_{\sigma_m}\in I(a^*)\}.\]
This proves the claim.
%It follows that\ $\mu^*$ is the optimal stopping time with regard to the optimal stopping problem for
%\[\rho^{\sigma_m}g(Y_{\sigma_m}),\;m\in\N_0,\]
%with respect to the filtration $(\C_m)_{m\in\N_0},\;\C_m=\A_{\sigma_m}$. 
\end{proof}

\begin{remark}\label{rem:reversed}
The arguments so far will apply if the function $g$ is 0 up to $b$ and, in our examples, is increasing after $b$. Now, consider the reversed situation. Let us assume that
\[\{y:g(y)>0\}=\{y:y<b\}\mbox{ for some }b\in\R\cup\{\infty\}.\]
Assume that we have successfully shown that
\[S^*=\{y:y\leq a^*\}\mbox{ or }=\{y:y< a^*\}\mbox{ for some }a^*<b.\]
This corresponds to situations where $g$ is decreasing up to $b$. Now, of course, we have to  look at the descending ladder variables defined via
\begin{align*}
\sigma_0&=\inf\{j:Y_j<b\},\;\;\;\sigma_0=0\mbox{ if }b=\infty,\\
\sigma_n&=\inf\{j>\sigma_{n-1}:Y_j<Y_{\sigma_{n-1}}\},\;n\geq 1,\;\sigma_\infty=\infty
\end{align*}
Then, the reasoning is as before and to find the optimal level $a^*$ we have to solve the optimal stopping problem for the descending ladder variables.
\end{remark}

\section{The Monotone Case Problem}\label{sec:monotone_case}
Most optimal stopping problems are not straightforwardly solved. One of the few exceptions is the 
class of monotone stopping problems as introduced already in \cite{MR0132593,MR0157465}. A long list of examples can be found in \cite{CRS} and, more recently, in \cite{Ferguson}. 

There is no hope that the stopping problem \eqref{eq:OSP} is monotone in most non-trivial situations. However, the stopping problem which we, as argued in Section \ref{sec:sec_step}, now have to solve is given by the reward process
\begin{align}\label{eq:monotone_stopping_prob}
Z_m=\rho^{\sigma_m}g(Y_{\sigma_m}),\;m\in\N_0,
\end{align}
with respect to $(\C_m)_{m\in\N_0}$. We want to take advantage of the monotonicity of the sequence $(Y_{\sigma_m})_{m\in\N_0}$ which hopefully will lead to monotone case optimal stopping problems. So, we look at the conditional expectation 
\[E(Z_{m+1}|\C_m)=E\left(\rho^{\sigma_{m+1}}g(Y_{\sigma_{m+1}})|\A_{\sigma_m}\right).\]
(Here we omit the starting point as the conditional expectation depends only on the value of $Y_{\sigma_m}$.) This has to be compared to 
\[Z_m=\rho^{\sigma_m}g(Y_{\sigma_m}).\]
The monotone case property states that for all $m\geq 0$
\[E(Z_{m+1}|\C_m)\leq Z_m\mbox{ implies }E(Z_{m+2}|\C_{m+1})\leq Z_{m+1}.\]
The myopic stopping time is given as
\[\inf\{m\geq0:E(Z_{m+1}|\C_m)\leq Z_m\}.\]

\begin{defn}\label{defi:phi}
	We write for all $y$
	\[\tau_y=\inf\{n>0:Y_n>y\},\;\phi(y)=E_y\rho^{\tau_y}g(Y_{\tau_y}),\;\;\phi(\Delta)=0.\]
\end{defn}

\begin{proposition}\label{prop:mon_case_markov}
	For all $m\geq 0$
	\[E(Z_{m+1}|\C_m)\leq Z_m\mbox{ iff }\phi(Y_{\sigma_m})\leq g(Y_{\sigma_m}).\]
\end{proposition}
\begin{proof}
On $\{\sigma_m=\infty\}$, we have 
\[E\left(\rho^{\sigma_{m+1}}g(Y_{\sigma_{m+1}})|\A_{\sigma_m}\right)=0=\rho^{\sigma_m}g(Y_{\sigma_m})\]
and
\[E\left(\rho^{\sigma_{l+1}}g(Y_{\sigma_{l+1}})|\A_{\sigma_l}\right)=0=\rho^{\sigma_l}g(Y_{\sigma_l})\] 
for all $l>m$. So, on $\{\sigma_m=\infty\}$ we trivially have the monotone case condition. We now argue on $\{\sigma_m<\infty\}$. Then, 
\[E\left(\rho^{\sigma_{m+1}}g(Y_{\sigma_{m+1}})|\A_{\sigma_m}\right)=E\left(\rho^{\sigma_{m+1}-\sigma_m}g(Y_{\sigma_{m+1}})|\A_{\sigma_m}\right)\rho^{\sigma_m}\]
and comparing it with $\rho^{\sigma_m}g(Y_{\sigma_m})$ we may ignore $\rho^{\sigma_m}$ as a common positive factor. Conditioning on $\A_{\sigma_m}$ we use the Markov property. Setting $y=Y_{\sigma_m}$ it follows
\[E\left(\rho^{\sigma_{m+1}-\sigma_m}g(Y_{\sigma_{m+1}})|\A_{\sigma_m}\right)=E_y(\rho^{\tau_y}g(Y_{\tau_y}))=\phi(y).\]
% 
%We set\ 
 It follows
\[E(Z_{m+1}|\C_m)\leq Z_m\mbox{ iff }\phi(Y_{\sigma_m})\leq g(Y_{\sigma_m}).\]
\end{proof}

Now, the monotonicity of the $Y_{\sigma_m}$ as ascending ladder variables comes into play, noting that $Y_{\sigma_m}>b$ on $\{\sigma_m<\infty\}$. 

\begin{proposition}
%	Write $\phi(y)=E_y(\rho^{\tau_y}g(Y_{\tau_y}))$ and
Assume that
\begin{align}\tag{M}\label{eq:M}
\phi(z)\leq g(z)\mbox{ implies }\phi(z')\leq g(z')\mbox{ for }z'>z>b.
\end{align}
Then,
\[\phi(Y_{\sigma_m})\leq g(Y_{\sigma_m})\mbox{ implies }\phi(Y_{\sigma_{m+1}})\leq g(Y_{\sigma_{m+1}})\mbox{ for all }m\geq 0\]
and the monotone case property holds for every $P_y$ in the stopping problem for \eqref{eq:monotone_stopping_prob}.
\end{proposition}

\begin{proof}
Let\ $\phi(Y_{\sigma_m})\leq g(Y_{\sigma_m})$. On $\{\sigma_{m+1}=\infty\}$, trivially
\[\phi(Y_{\sigma_{m+1}})=0= g(Y_{\sigma_{m+1}}).\]
On $\{\sigma_{m+1}<\infty\}$ we have $\sigma_m<\infty$, hence
\[\phi(Y_{\sigma_m})\leq g(Y_{\sigma_m})\mbox{ and }b<Y_{\sigma_m}<Y_{\sigma_{m+1}}\]
implying by \eqref{eq:M} the inequality $\phi(Y_{\sigma_{m+1}})\leq g(Y_{\sigma_{m+1}}).$ Using Proposition \ref{prop:mon_case_markov} this shows the monotone case property.
\end{proof}

So, under \eqref{eq:M}, we have a monotone stopping problem with myopic stopping time
\[\nu^*=\inf\{m\geq 0:\phi(Y_{\sigma_m})\leq g(Y_{\sigma_m})\}.\]
Under well-known conditions, for example under \eqref{eq:Opt}, $\nu^*$ is optimal, and the optimal level of the original problem is given by
\[a^*=\inf\{z:\phi(z)\leq g(z)\}.\]
 This result can be found in the references given above. A recent treatment -- based on the Doob decomposition -- is given in the Appendices to \cite{ChristensenIrle17}. 
 
%We now summarize our findings:
%\begin{proposition}
%	Assume that an optimal stopping time in \eqref{eq:OSP} is of threshold type
%	and \eqref{eq:M} holds true.
%	Then the optimal threshold $a^*$ for 
%		\[\tau^*:=\inf\{n:Y_n\geq a^*\}\]
%		is given by
%	\[a^*=\inf\{z:\phi(z)\leq g(z)\}.\]
%\end{proposition}

\begin{thm}\label{thm:opt_threshold}
	Assume \eqref{eq:Opt}, that the optimal stopping set for problem \eqref{eq:OSP} is $S^*=[a^*,\infty)$ or $S^*=(a^*,\infty)$, and assume validity of \eqref{eq:M}. Then 
	\[a^*=\inf\{z>b:\phi(z)\leq g(z)\}.\]
	If $\phi(a^*)\leq g(a^*)$, then $S^*=[a^*,\infty)$, otherwise $S^*=(a^*,\infty)$.
\end{thm}
\begin{proof}
	For all\ $z>a^*$ we have $V(z)=g(z)$, hence $\phi(z)\leq V(z)\leq g(z)$, so that
	\[\alpha^*:=\inf\{z>b:\phi(z)\leq g(z)\}\leq a^*.\]
	Now, $\{z>b:\phi(z)\leq g(z) \}=I(\alpha^*),$ and the myopic stopping time 
	\[\nu^*=\inf\{m\geq 0:Y_{\sigma_m}\in I(\alpha^*)\} \]
	is optimal for $(Y_{\sigma_m})_m,$ hence $\sigma^*=\sigma_{\nu^*}$ is optimal in \eqref{eq:OSP} for every $P_y$.\\
	Now assume that $\alpha^*<a^*$ and choose $\alpha^*<y<a^*.$ Then $\sigma^*=0$ is optimal for $P_y$ which contradicts $g(z)<V(z)$ for $z<a^*$.\\
	If $\phi(a^*)\leq g(a^*)$, then $\{z>b:\phi(z)\leq g(z)\}=[a^*,\infty)$, and for $y=a^*$, $\sigma^*=0$ is optimal for $P_{a^*}$, hence $a^*\in S^*$. If $\phi(a^*)>g(a^*)$, then $V(a^*)>g(a^*)$, so that $a^*\not\in S^*$.
\end{proof}

\begin{remark}\label{rem:ratio}
	Note that\ \eqref{eq:M} follows if 
	$\phi/g$ is decreasing on $\{y:y>b\}$ since $g>0$ on this set. From a computational point of view, this is an easy-to-handle situation as the underling function is monotone.
\end{remark}

\begin{remark}
	In the case that $g$ and $\phi$ are continuous, $a^*$ fulfils
	\[E_{a^*}(\rho^{\tau_{a^*}}g(Y_{\tau_{a^*}}))=g(a^*).\]
	Note that this can be interpreted as a continuous fit condition, see \cite{ps} for a discussion of the continuous- and smooth-fit-principle.
\end{remark}

%\section{On the first step}
%We have proved that under \eqref{eq:M} and [...], the stopping time 
%\[\tau^*=\inf\{n:Y_n\geq a^*\}\]
%is optimal for \eqref{eq:monotone_stopping_prob}, where 
%\[a^*=\sup\{y:\frac{\phi(y)}{g(y)}\geq 1\}\] and if we have carried out the first step, then $\tau^*$ is also optimal for the original problem \eqref{eq:OSP}.
%We now show that the slightly stronger assumption  
%\begin{align}\tag{M1}\label{eq:M1}
%\mbox{There exists }a^*\mbox{ s.t. } 
%\frac{\phi(y)}{g(y)}\geq 1\mbox{ for }y\leq a^*\mbox{ and }\frac{\phi(y)}{g(y)}\mbox{ is decreasing for }y\geq a^*,
%\end{align}
%directly implies that $\tau^*$ is optimal for the original problem. As $\tau^*$ 
%gives higher reward in \eqref{eq:monotone_stopping_prob} than immediate stopping and $\tau^*$ gives the same reward in \eqref{eq:OSP}, we have that $(-\infty,a^*)\subseteq \R\setminus S^*.$ It remains to prove that immediate stopping is optimal when the processes starts in $y>a^*.$

\section{A sufficient condition for optimality of threshold times}
The main aim of this section is to give a sufficient criterion in terms of the function\ $\phi$ to warrant that a threshold time 
%for 
%	\[\alpha^*=\inf\{z:\phi(z)\leq g(z)\}\]
	 is optimal for original problem \eqref{eq:OSP}. 
	 As discussed before, there are various results of an elementary nature which show this in special situations, see also the Appendix. Here we shall give a general result for this optimality in terms of the function\ $\phi$ introduced above. 
	 
	 The candidate for the optimal threshold is 
	 \[\alpha^*=\inf\{z:g(z)\geq \phi(z)\}.\]
	 We shall show in the following that the threshold time for the level $\alpha^*$ is in fact optimal under general assumptions, so the $\alpha^*$ will in fact be the $a^*$ of the foregoing parts. 
	 We define
	 \begin{equation}\label{eq:f_def}
	 f(y)=\frac{1}{1-E_y(\rho^{\tau_y})}(g(y)-\phi(y)).
	 \end{equation}
	 For $0<\rho<1$, we have $0<E_y(\rho^{\tau_y})\leq 1$, for $\rho=1$, and using the convention $\rho^\infty=0$ for $\rho=1,$ we have $1-E_y(\rho^{\tau_y})=P_y(\tau_y=\infty)$. For $f$ to be well-defined, we therefore assume for $\rho=1$ that $P_y(\tau_y=\infty)>0$ for all\ $y$. 
	  Furthermore note that $f(y)<0$ for $y<\alpha^*$ by definition of $\alpha^*$.
%	where we assume that $P(\tau_y<\infty)>0$ for all\ $y$ for $f$ to be well-defined. 
	 
	 Obviously, \eqref{eq:M} can be reformulated as 
\[f(z)\geq 0\mbox{ implies }f(z')\geq0 \mbox{ for }z'>z>b.\]
Now $\alpha^*=\inf\{z:f(z)\geq 0\}$, so we may have one, or both for $f(\alpha^*)=0$, of the following two cases:
\begin{align}
f(\alpha^*)&\geq 0\tag{A1}\label{eq:A1}\\
f(\alpha^*)&\leq 0\tag{A2}\label{eq:A2}
\end{align}

\begin{thm}\label{thm:optimal}
	Assume \eqref{eq:Opt} and for $\rho=1$ that $P_y(\tau_y=\infty)>0$ for all $y$.
	Assume that 
		 \begin{align}\tag{M1}\label{eq:M1}
%	f(y)\leq 0\mbox{ for $y\leq \alpha^*$ and 
\mbox{$f$ is 
%	positive\footnote{ist $f\geq 0$ ok?} and 
increasing for $y\geq \alpha^*$.}
	\end{align}
	\begin{enumerate}[(i)]
		\item Assume \eqref{eq:A1}. Then
		\[\tau^*=\inf\{n\geq 0:Y_n\geq \alpha^*\}\]
		is an optimal stopping time for problem \eqref{eq:OSP}, i.e., $a^*=\alpha^*$.
		\item Assume \eqref{eq:A2}. Then
		\[\tau^*=\inf\{n\geq 0:Y_n>\alpha^*\}\]
		is an optimal stopping time for problem \eqref{eq:OSP}.
	\end{enumerate}
%Then, 
%	\[\tau^*:=\inf\{n:Y_n> a^*\}\]
%	is the optimal stopping time for problem \eqref{eq:OSP}. 
\end{thm}

\begin{proof}
	\begin{enumerate}[(i)]
		\item By assumption, the condition \eqref{eq:M} holds, and we have the monotone case problem \eqref{eq:monotone_stopping_prob} with value function
	\[\tilde V(y):=\begin{cases}
g(y),&y\geq \alpha^*,\\
E_y\rho^{\tau^*}g(Y_{\tau^*}),&y<\alpha^*.
\end{cases}\]
With
 \[\tau'=\inf\{n>0:Y_n\geq \alpha^*\}\]
 we have
 \[\tilde V(y)=E_y\rho^{\tau'}g(Y_{\tau'})\mbox{ for }y<\alpha^*.\]
 Due to the optimality in problem \eqref{eq:monotone_stopping_prob}, we have
 \[\tilde V(y)\geq g(y)\mbox{ for all }y,\]
 so that $\tilde V$ is a majorant of $g$. By the general theory (or a short elementary argument), it remains to prove that $\tilde V$ is $\rho$-excessive for $Y$, that is 
 \[E_y(\rho \tilde V(Y_1))\leq \tilde V(y)\mbox{ for all }y.\]
% Note that we have equality for  $y<\alpha^*$, so that we assume now  $y\geq \alpha^*$.
 By the strong Markov property
 \[E_y(\rho \tilde V(Y_1))=E_y\rho^{\tau'}g(Y_{\tau'})\mbox{ for all }y,\]
 so equality $E_y(\rho \tilde V(Y_1))= \tilde V(y)\mbox{ holds for all }y<\alpha^*$. So let $y\geq \alpha^*$. 	Write 
 \begin{align*}
 \kappa_0&=\tau'=\inf\{j>0:Y_j\geq \alpha^*\},\\
 \kappa_n&=\inf\{j>\kappa_{n-1}:Y_j>Y_{\kappa_{n-1}}\},\;n\geq 1.
 \end{align*}
 	Using this, we obtain
 \begin{align*}
 E_y\rho^{\tau'}g(Y_{\tau'})-\phi(y)&= E_y(\rho^{\tau'}g(Y_{\tau'})-\rho^{\tau_{y}}g(Y_{\tau_y}))1_{\{\tau'<\tau_y\}}\\
 &=\sum_{n=0}^\infty E_y(\rho^{\kappa_n}g(Y_{\kappa_n})-\rho^{\kappa_{n+1}}g(Y_{\kappa_{n+1}}))1_{\{\kappa_n<\tau_y\}}	\\
 &=\sum_{n=0}^\infty E_y\rho^{\kappa_n}(g(Y_{\kappa_n})-E_y(\rho^{\kappa_{n+1}-\kappa_n}g(Y_{\kappa_{n+1}})|\mathcal{A}_{\kappa_n}))1_{\{\kappa_n<\tau_y\}}	\\
 &=\sum_{n=0}^\infty E_y\rho^{\kappa_n}(g(Y_{\kappa_n})-\phi({Y_{\kappa_n}}))1_{\{\kappa_n<\tau_y\}}\\
 &\leq \sum_{n=0}^\infty E_y\rho^{\kappa_n}(g(y)-\phi(y))\frac{1-E_{Y_{\kappa_n}}(\rho^{\kappa_1})}{1-E_y(\rho^{\kappa_1})}1_{\{\kappa_n<\tau_y\}},
 \end{align*}
 where we used $\{\kappa_n<\tau_y\}\in\A_{\kappa_n}$ in the third step and  \eqref{eq:M1} in the last step. Note that this also holds on $\{\tau_y=\infty\}$ by \eqref{eq:Opt}.
 		Now, we introduce an additional random variable $T$ independent of everything else with geometric distribution with parameter $1-\rho$, so that $P(T>n)=\rho^n$ for all $n$. In the case $\rho=1$ we use $T\equiv\infty$. We will make use of the memoryless-property of $T$ in the following calculation. 	
 		\begin{align*}
 	E_y\rho^{\tau'}g(Y_{\tau'})-\phi(y)
 	&\leq (g(y)-\phi(y))\frac{1}{P_y(\kappa_1\geq T)}\sum_{n=0}^\infty E_y(1_{\{\kappa_n<T\}}P_y(\kappa_{n+1}\geq T|\A_{\kappa_n})1_{\{\kappa_n<\tau_y\}})\\
 	&=(g(y)-\phi(y))\frac{1}{P_y(\tau_y\geq T)}\sum_{n=0}^\infty P_y(\kappa_n<T\leq \kappa_{n+1}\leq \tau_y)\\
 	&=q(g(y)-\phi(y))
 	\end{align*}
 	with $q=P_y(\kappa_0<T|\tau_y\geq T)\in[0,1]$. Hence, using $\phi(y)\leq g(y)$ for $y\geq \alpha^*$,
 	\[E_y\rho^{\tau'}g(Y_{\tau'})\leq q g(y)+(1-q)\phi(y)\leq g(y).\]
 	\item Under \eqref{eq:A2} simply use 
 	\[\tau'=\inf\{n>0:Y_n>\alpha^*\}=:\nu_0\]
 	and look at $y\leq \alpha^*$ and $y>\alpha^*$. Then, the proof works in the same way. 
	\end{enumerate}

\end{proof}

%Some questions:
%\begin{enumerate}
%	\item Is it clear that the representation of $g$ holds if we define $f$ via \eqref{eq:f_def}?
%	\item Can we perhaps find other assumptions on $g,\phi$ to guarantee optimality of threshold times for the original problem \eqref{eq:OSP}? [Monotonicity of $g/\phi$? would be perfect, but I do not see why this should hold true.]
%\end{enumerate}

The following criterion for \eqref{eq:M1} turns out to be useful.
\begin{lemma}\label{lem:suff_M_1}
On the set $[\alpha^*,\infty)$ let the following hold true:
\begin{itemize}
	\item $y\mapsto E_y(\rho^{\tau_y})$ is increasing,
	\item $\phi/g$ is decreasing,
	\item $\phi$ is increasing or $g$ is increasing.
\end{itemize}
% $\phi/g$ is decreasing and $\phi$ and $y\mapsto E_y(\rho^{\tau_y})$ are increasing. \\
 Then, \eqref{eq:M1} holds true.
% the function $f$ defined via \eqref{eq:f_def} is increasing on $[\alpha^*,\infty)$.
\end{lemma}

\begin{proof}
%	It holds that
%	\begin{align*}
%	\phi'(x)-g'(x)&=\phi'(x)\frac{g(x)}{\phi(x)}-g'(x)+\phi'(x)\left(1-\frac{g(x)}{\phi(x)}\right)\\
%	&=\left(\frac{\phi}{g}\right)'(x){\phi(x)}+\phi'(x)\left(1-\frac{g(x)}{\phi(x)}\right)\\
%	&\leq\phi'(x)\left(1-\frac{g(x)}{\phi(x)}\right)\\
%	&\leq 0,
%	\end{align*}
%	where we used that ${g(x)}\geq {\phi(x)}$ for $x\geq \alpha^*$ (by monotonicity). 
Writing
\[g-\phi=g(1-\phi/g)\mbox{ and }g-\phi=\phi(g/\phi-1)\]
and using that ${g(x)}\geq {\phi(x)}$ for $x\geq \alpha^*$, we obtain that $g-\phi$ is increasing on $[\alpha^*,\infty).$
	Hence, the numerator in
	\[f(x)=\frac{1}{1-E_x(\rho^{\tau_x})}(g(x)-\phi(x)),\]
	is increasing. As the denominator is decreasing by assumption, this proves the result.
\end{proof}

\section{The Random Walk case}\label{sec:RW}
%Assume in this section that 
The form of $\phi$ simplifies for the case of a random walk. Let $X_1,X_2,\dots$ be an i.i.d. sequence and $S_n=\sum_{i=1}^{n}X_i,\,S_0=0$, so that we work with
\[Y_n^y=y+S_n\]
as Markov process with starting point $y.$ Here, $P_y$ is the distribution of $Y^y$ with respect to an underlying probability measure $P$. To avoid trivial cases, we assume $P(X_1>0)>0.$ Then,
\[\tau_y=\inf\{n>0:y+S_n>y\}=\inf\{n>0:S_n>0\}=\tau_+,\]
say, and we obtain
\[\phi(y)=E_y\rho^{\tau_y}g(Y_{\tau_y})=E\rho^{\tau_+}g(y+S_{\tau_+}).\]
%where we write $E(\cdot)=E_0(\cdot)$ for short.
To apply Theorem \ref{thm:optimal}, we check the conditions of Lemma \ref{lem:suff_M_1}. The previous considerations (applied to $g\equiv 1$) yield
\[E_y\rho^{\tau_y}=E\rho^{\tau_+},\]
being independent of $y$. Furthermore,
\[\frac{\phi(y)}{g(y)}=E\rho^{\tau_+}\frac{g(y+S_{\tau_+})}{g(y)}.\]
%Keeping Remark \ref{rem:ratio} in mind, we obtain that \eqref{eq:M} holds true if for all $s\geq 0$
This shows that $\phi/g$ is decreasing on $[\alpha^*,\infty)$ if
\begin{equation*}
\frac{g(y+s)}{g(y)}\mbox{ is decreasing in $y$ on $[\alpha^*,\infty)$},
\end{equation*}
i.e. for all $s\geq 0$
\begin{equation*}
{\log g(y+s)}-\log {g(y)}\mbox{ is decreasing in $y$ on $[\alpha^*,\infty)$},
\end{equation*}
which is, $\log g$ is concave on $[\alpha^*,\infty)$, i.e. $g$ is $\log$-concave on $[\alpha^*,\infty)$.
%If $g$ is differentiable\footnote{is perhaps not needed}, then this means that for all $y>b$, $s\geq 0$
%\[0\geq \left(\frac{g(y+s)}{g(y)}\right)'=\frac{g'(y+s)g(y)-g(y+s)g'(y)}{g(y)^2},\]
%i.e.
%\[\frac{g'(y+s)}{g(y+s)}\leq \frac{g'(y)}{g(y)},\]
%which is, $g$ is log-concave on $(b,\infty)$.
% We obtain:
%\begin{lemma}
%	If $Y$ is a random walk and $g$ is  log-concave on $\{y:y>b\}$, then
%	$\phi/g$ is decreasing on $\{y:y>b\}$.
%	 In particular,
%	\eqref{eq:M} holds true.
%\end{lemma}
%
%Note that in the random-walk case, 
%\[E_x(\rho^{\tau_x})=\E(\rho^{\tau_+})\]
% is constant in $x$. Furthermore, if $g$ is increasing, so is 
%\[\phi(x)=E(\rho^{\tau_+}g(x+Y_{\tau_+})).\] 
Therefore, Lemma \ref{lem:suff_M_1} and Theorem \ref{thm:optimal} yield

\begin{thm}\label{prop:RW_optimal}
	If $Y$ is a random walk, then
	\[\alpha^*=\inf\left\{z>b:E\rho^{\tau_+}\frac{g(z+S_{\tau_+})}{g(z)}\leq 1\right\}.\]
	 Assume now that $g$ is increasing and log-concave on $\{y:y>\alpha^*\}$ and \eqref{eq:Opt} holds true.
%	Assume \eqref{eq:Opt} and for $\rho=1$ that $P_y(\tau_y=\infty)>0$.
%	Assume that 
%	\begin{align}\tag{M1}\label{eq:M1}
%	%	f(y)\leq 0\mbox{ for $y\leq \alpha^*$ and 
%	\mbox{$f$ is 
%		%	positive\footnote{ist $f\geq 0$ ok?} and 
%		increasing for $y\geq \alpha^*$.}
%	\end{align}
	\begin{enumerate}[(i)]
		\item Assume \eqref{eq:A1}. Then
		\[\tau^*=\inf\{n\geq 0:Y_n\geq \alpha^*\}\]
		is an optimal stopping time for problem \eqref{eq:OSP}, i.e. $a^*=\alpha^*$.
		\item Assume \eqref{eq:A2}. Then
		\[\tau^*=\inf\{n\geq 0:Y_n>\alpha^*\}\]
		is an optimal stopping time for problem \eqref{eq:OSP}.
	\end{enumerate}
\end{thm}

\begin{remark}
\cite{hsiau2014logconcave} considered optimal stopping problems for skip-free random walks on $\Z$ (and, in continuous time, L\'evy processes with no upward jumps). The main finding of their paper was a variant of the previous proposition in this special case. Their method of proof relies heavily on the fact that no overshoot occurs in the skip-free case and is therefore structurally different from ours.\\
They even proved that these assumptions on $g$ are necessary in the following sense: If the optimal stopping time for $g:\Z \rightarrow [0, \infty)$ is of threshold type for all Bernoulli random walks, then $g$ is increasing and log-concave. This, however, shows that the previous theorem gives a characterization of all reward functions that lead to one-sided stopping situations for all random walks (and all discount factors).
\end{remark}

% \begin{corollary}
%...
% \end{corollary}

\section{Examples}\label{sec:ex}
We firstly consider two prominent and one new example for underlying random walks. After this,  starting from i.i.d. sequences, we illustrate our results for more general Markovian sequences.
%\begin{example}
%[The Novikov-Shiryaev problem]
\subsection{The Novikov-Shiryaev problem}
We consider here and in the following two examples in the setting of Section \ref{sec:RW}. For $\nu>0$, we assume finiteness of $E|X_i|^\nu$ and set 
\[g(y)=(y^+)^\nu,\]
so that 
\[V(y)=\sup_{\tau}E\rho^\tau ((y+S_\tau)^+)^\nu,\;\;\rho\in(0,1].\]
The condition \eqref{eq:Opt} for optimality of $\tau^*$ holds if
\[\rho=1,\,E(X_1)<0,\, E(X_1^+)^{\nu+1}<\infty\mbox{ or }\rho<1,\, E(X_1^+)^{\nu}<\infty, \]
see \cite{NS2}, Lemma 3. It is immediately seen that $g$ is log-concave on $[0,\infty)$, so that Theorem \ref{prop:RW_optimal} yields the optimality of $\inf\{n\geq 0:X_n\geq a^*\}$, where $a^*$ is uniquely given by
\[a^*=a^*(\nu)=\inf\left\{y>0:E\rho^{\tau_+}\left(1+\frac{S_{\tau_+}}{y}\right)^\nu\leq 1\right\}.\]
In the non-trivial situation $P(X_i>0)>0$, we have that $a^*>0$ is the positive solution to
\[p_\nu\left(\frac{1}{y}\right)=E\rho^{\tau_+}\left(1+\frac{S_{\tau_+}}{y}\right)^\nu=1.\]
In the case $\nu$ is an integer, it is clear that $p_\nu$ is a polynomial of degree $\nu$ and the coefficients are given explicitly in terms of moments: $\binom{\nu}{m} E\rho^{\tau_+}{S_{\tau_+}^m},\,m\leq \nu.$

In \cite{NS} and \cite{NS2}, the optimal level was obtained using Appell polynomials by far more sophisticated techniques. The Appell polynomials $Q_\nu$ can be  found recursively via $Q_0(x)=1$ and
\[Q_{\nu}'(x)=\nu Q_{\nu-1}(x),\;EQ_{\nu}(M_T)=0,\;\nu=1,2,\dots,\]
where, as above, $M_T$ denotes the running maximum process evaluated at an independent geometric time $T$ with parameter $1-\rho$. The optimal threshold is then found to be the root of $Q_\nu$. Using the results discussed in Subsection \ref{subsec:other_approaches} below, it follows from the discussion in \cite{CST}, Subsection 3.3, that $Q_\nu$ coincides with our function $f$ from \eqref{eq:f_def}. As the roots of $f(y)=\frac{1}{1-E_y(\rho^{\tau_y})}(g(y)-\phi(y))$ and solutions to $\phi(y)/g(y)=E\rho^{\tau_+}\left(1+\frac{S_{\tau_+}}{y}\right)^\nu=1$ coincide, both results are indeed equivalent. 
%\end{example} 

%\begin{example}
	\subsection{American Put}
	For $K>0,\,\rho\in(0,1)$ set
	\[g(y)=(K-\exp(y))^+\]
	and consider the associated problem with value function
	\[V(y)=\sup_{\tau}E\rho^\tau(K-\exp(Y^y_\tau))^+. \]%\footnote{$EX_1<0$ needed?}
	Note that we could equivalently use a multiplicative random walk on $(0,\infty)$ with reward $y\mapsto(K-y)^+$.
	Condition \eqref{eq:Opt} is obviously fulfilled and the reward $g$ is decreasing and the second logarithmic derivative is 
	\[-K\exp(y)(K-\exp(y))^{-2}<0,\]
	so that $g$\ is log-concave on $(-\infty,\log(K)]$. We can therefore apply Theorem \ref{prop:RW_optimal} in the reversed situation described in\ Remark \ref{rem:reversed}. This yields optimality of 
	\[\tau^*=\inf\{n\geq 0:Y_n\leq  a^*\}\]
	with $a^*$ given by 
\[a^*=\inf\left\{y>0:E\rho^{\tau_-}\left(\frac{K-\exp(y+S_{\tau_-})}{K-\exp(y)}\right)\leq 1\right\}.\]	
Here, in the spirit of Remark \ref{rem:reversed}, 
\[\tau_-=\inf\{m>0:S_m<0\}.\]
More explicitly, we have
\[a^*=\log\left(K\frac{1-E(\rho^{\tau_-})}{1-E(\rho^{\tau_-}e^{S_{\tau_-}})}\right).\]
The Pecherskii-Rogozin identity for random walks, see \cite{Kyp_Wiener_Hopf}, easily yields the alternative representation
\[a^*=\log\left(KE\exp(I_T)\right),\]
where $I$ denotes the running minimum process and $T$ is a time as before, reproducing the result in \cite{Mo}. We refer to \cite{CIN} and \cite{baurdoux2013direct} for showing elementary that threshold times are optimal. 
%\end{example}

 \subsection{Logistic reward}\label{subsec:logistic}
	We want to underline that the theory developed here is applicable directly also for new classes of rewards for an underlying random walk. Indeed, consider a logistic reward function of the form
		\[g(y)=\frac{1}{K\exp(-\eta y)+1}\]
		for some $K,\eta>0.$ In the particular case of a spectrally negative L\'evy process, this problem was studied in \cite{RePEc:tkk:dpaper:dp9}. Similar to the previous examples, Theorem \ref{prop:RW_optimal} yields that a threshold time is optimal and the optimal threshold is given as the unique solution $a^*$ to
		\[E\rho^{\tau_+}\frac{Ke^{-\eta y}+1}{Ke^{-\eta(y+S_{\tau_+})}+1}=1.\]
		A more explicit expression in the particular case of exponential upward jumps is presented at the end of Subsection \ref{subsec:exponential_innov} below.

%\section{Further examples}
\subsection{Shepp-Shiryaev Russian option problem}\label{subsec:Shepp_shi}
Let $X_1,X_2,\dots$ be an i.i.d. sequence and $S_n=\sum_{i=1}^nX_i,\;S_0=0.$ We consider the optimal stopping problem for
\[\rho^n\exp(\max\{y,S_0,S_1,\dots,S_n\}),\;n=0,1,2,\dots,\,\rho\in(0,1),\;y\geq 0, \]
so that
\[V(y)=\sup_{\tau} E\rho^\tau\exp(\max\{y,S_0,S_1,\dots,S_\tau\}). \]
The continuous time problem was introduced by \cite{SS} and \cite{SheppShir94}. The general discrete time setting for this problem as given in this paper seems to be new. 

We use the change of measure approach of \cite{SheppShir94} as follows. Write $\mu=Ee^{X_1}$ and
\begin{align*}
&\rho^n\exp(\max\{y,S_0,S_1,\dots,S_n\}-S_n)\exp(S_n)\\
=&(\rho\mu)^n\exp(\max\{y,S_0,S_1,\dots,S_n\}-S_n)\prod_{i=1}^n\frac{\exp(X_i)}{\mu}. 
\end{align*}
Now, we use the probability measure 
\[\frac{dQ}{dP}\bigg|_{\A_n}=\prod_{i=1}^n\frac{\exp(X_i)}{\mu},n=1,2,\dots .\]
With respect to $Q$, the random variables $X_1,X_2,\dots$ are again i.i.d. with\ $E_QX_i=\frac{1}{\mu}EX_i\exp(X_i).$ 
%We assume that $\rho\mu<1.$ 
Then, we arrive at an optimal stopping problem w.r.t. $Q$\ and have 
\[V(y)=\sup_{\tau} E_Q(\rho\mu)^\tau\exp(\max\{y,S_0,S_1,\dots,S_\tau\}-S_\tau). \]
Using the notation $r=\rho\mu,\;Y_n^y=\max\{y,S_0,S_1,\dots,S_n\}-S_n,$ we have
\[V(y)=\sup_{\tau} E_Qr^\tau\exp(Y^y_\tau). \]
Here
\[Y_0^y=y\geq 0,\,Y^y_{n+1}=(Y_n^y-X_{n+1})^+,\]
so we have a Markov process with starting point\ $y\in[0,\infty)$ and random variable representation $Y^y_{n+1}=(Y_n^y-X_{n+1})^+$ with respect to $Q$.
Although variants of this problem have been studied extensively, the question whether condition \eqref{eq:Opt} is fulfilled or not does not seem to have been addressed. Because it seems to be of some more general interest, we discuss this in Appendix \ref{sec:appendix_opt_ss_problem}. It turns out that \eqref{eq:Opt} holds under the natural condition
\[r=\rho\mu<1,\] which is assumed in the following.

The first step for this problem is elementary, see \cite{baurdoux2013direct},\ Section 4, and our Appendix \ref{sec:appendix_illust}, or also \cite{AAP}:\ There exists $a^*\geq 0$ such that 
\[S^*=[a^*,\infty)\]
and the first entrance time into $S^*$ is optimal. 
Note that\ $g=\exp>0$ on the whole state space $[0,\infty).$ If $a^*=0$, then immediate stopping is optimal, so let us assume $a^*>0$ in the following. For a starting point\ $y$
\[\tau_y=\inf\{m>0:Y_m^y>y\} \]
and we obtain
\begin{align*}
\phi(y)&=E_Qr^{\tau_y}\exp(Y^y_{\tau_y}),\\
\frac{\phi(y)}{g(y)}&=E_Qr^{\tau_y}\exp(Y^y_{\tau_y}-y).
\end{align*}

We now try to establish property \eqref{eq:M} by making use of Remark \ref{rem:ratio}. First, let\ $y'>y\geq 0$ and consider corresponding random paths starting in $y',\,y.$ Then, the $y'$-path is always above the $y$-path and the differences at any time point between these paths is $\leq y'-y$ (with equality as long as the $0$-boundary is not reached by the $y$-path). So, if the $y'$-path is $>y'$, then the $y$-path is $>y$, which shows
\begin{equation}\label{eq:SS_coupling}
\tau_y\leq \tau_{y'}.
\end{equation}
In particular
\[E_Qr^{\tau_{y'}}\leq E_Qr^{\tau_{y}}.\]

We first only look at the special case that $X_m$ takes values in $\Z$, with state space of $Y_m$ being $\N\cup\{0\},$ and that the positive jumps of\ $Y_m$ can only take the value $+1$, this being equivalent to $X_i=-1$ on $\{X_i<0\}$. For $y\in \N$, we have 
\[Y_{\tau_y}-y=1\mbox{ on }\{\tau_y<\infty\},\]
hence
\[\frac{\phi(y)}{g(y)}=E_Q(r^{\tau_y}\exp(1))\]
is decreasing in $y$ due to \eqref{eq:SS_coupling}. Therefore, we have the monotone case with optimal
\[a^*=\inf\{y:E_Q(r^{\tau_y})\leq \exp(-1)\}.\]
The result can be seen as a generalization of the result in \cite{MR1348194} with a much simplified derivation. 
The following proposition shows that the monotone case holds in more general situations as well.

\begin{proposition}
	Assume that the positive jumps of $Y_m$ are bounded by some constant $B>0,$ i.e. $X_m^-\leq B$. Then, for $r<\exp(-B)$, the monotone case holds and 
	\[a^*=\inf\{y>0:E_Qr^{\tau_{y}}\exp(Y^{y}_{\tau_{y}}-y)\leq 1\}.\]
\end{proposition}

\begin{proof}
Consider $y'>y\geq 0$ and argue with the paths as above. We have
	\begin{align*}
	\frac{\phi(y')}{g(y')}&=E_Qr^{\tau_{y'}}\exp(Y^{y'}_{\tau_{y'}}-y')\\
	&\leq E_Q1_{\{\tau_{y'}=\tau_y\}}r^{\tau_{y}}\exp(Y^{y'}_{\tau_{y'}}-y')+E_Q1_{\{\tau_{y'}>\tau_y\}}r^{\tau_{y}+1}\exp(Y^{y'}_{\tau_{y'}}-y')\\
	&\leq E_Q1_{\{\tau_{y'}=\tau_y\}}r^{\tau_{y}}\exp(Y^{y}_{\tau_{y}}-y)+E_Q1_{\{\tau_{y'}>\tau_y\}}r^{\tau_{y}}re^B\exp(Y^{y}_{\tau_{y}}-y)\\
	&\leq E_Qr^{\tau_{y}}\exp(Y^{y}_{\tau_{y}}-y)=\frac{\phi(y)}{g(y)},
	\end{align*}
	where we have used \eqref{eq:SS_coupling}, $Y^{y'}_{\tau_{y'}}-y'\leq Y^{y}_{\tau_{y}}-y$, and $Y^{y'}_{\tau_{y'}}-y'\leq B$. 
\end{proof}

\subsection{Autoregressive processes with exponential upward innovations}\label{subsec:exponential_innov}
In addition to the i.i.d. sequence $X_1,X_2,\dots$ (called innovations), we furthermore fix a constant $0<\lambda<1$ and consider an autoregressive process of order 1 (AR(1)-process) $(Y_n)_{n\in\N_0}$ by 
\[Y_n=\lambda Y_{n-1}+X_n\mbox{ for all $n\in\N$}.\]
The random walk case $\lambda=1$ is excluded to avoid the distinction of different cases in the following. The Markovian structure can be identified from the decomposition
\[Y_n=\lambda^nY_0+\sum_{k=0}^{n-1}\lambda^kX_{n-k}\]
and we write $P_y$ for the distribution of $Y$ started in $Y_0=y$. As mentioned before, optimal stopping problems for this class of processes were studied in \cite{CIN}, see also \cite{NK} and \cite{F2}. In particular, it can be shown elementary that the following reward functions lead to optimal one-sided stopping times (under suitable integrability assumptions):
\[g(x)=(x-K)^+,\,K>0,\;g(x)=\left(x^+\right)^n,\,n\in\N.\]
Further examples can be identified for certain subclasses of innovation-distributions, in particular for processes with positive innovations. 
To invoke Theorem \ref{thm:opt_threshold} for finding the optimal threshold, we have to investigate the validity of condition \eqref{eq:M} in the autoregressive situation. This is a nontrivial problem since $\phi$ depends on the joint distribution of the ladder heights and -variable, and
%To go the second step, i.e. to find the optimal threshold $a^*$, we may apply Theorem \ref{thm:opt_threshold}. To this end, 
we have to compute $\phi(y)=E_y\rho^{\tau_y}g(Y_{\tau_y})$ as explicitly as possible. A class of AR(1)-processes where this is possible is given for innovations of the form
\begin{equation}\label{eq:phase_jumps}
X_i=X^+_i-X^-_i,
\end{equation}
where $X^\pm_i$ are independent, non-negative and $X^+_i$ has a distribution of phase-type, see \cite{AAP} for a corresponding result for L\'evy processes and \cite{christensen2012phase} for general AR(1)-processes. It is worth noting that this class of innovations
 is dense in the class of all distributions. To not overburden this paper, we just illustrate the approach in the easier case that $X_i^+$ is exponentially distributed with parameter $\mu$. As a special case of \cite[Corollary 1]{christensen2012phase}, it then holds that $\tau_y$ and the overshoot $Y_{\tau_y}-y$ are independent and $Y_{\tau_y}-y$ is exponentially distributed with the same parameter $\mu$. This yields
\begin{align}\label{eq:exp_innov_decomp}
\phi(y)=E_y\rho^{\tau_y}\cdot Eg(y+X_1^+)=E_y\rho^{\tau_y}\cdot \int_0^\infty g(y+x)\mu e^{-\mu x}dx.
\end{align}
From the pathwise representation of autoregressive processes it is clear that for $0<y\leq y'$ we have $\tau_y\leq \tau_{y'}$ and $E_y\rho^{\tau_y}\geq E_{y'}\rho^{\tau_{y'}}$. This shows that for $g$ as above, condition \eqref{eq:M} holds if 
\[y\mapsto Eg(y+X_1^+)\frac{1}{g(y)}\mbox{ is decreasing in }y>0,\]
and then the optimal threshold is given by
\[a^*=\inf\{y:E_y\rho^{\tau_y}\geq g(y)Eg(y+X_1^+)^{-1}\}\]
To obtain a more explicit expression for $a^*$, we have to determine $E_y\rho^{\tau_y}$. 
%All that is needed to apply the second step of our approach directly for general reward $g$ is now to determine $E_y\rho^{\tau_y}$.
 To obtain this, \cite[Theorem 3.2]{christensen2012phase} and \cite[Theorem 3.3]{CIN} can be applied. Using the notations $\exp(\psi),\,\exp(\psi_2)$ for the Laplace-transform of $X_1,$ $X_1^-$, resp., and 
\begin{equation*}
\eta(u)=\sum_{k=0}^\infty\psi(\lambda^k u),
\end{equation*}
we obtain the following explicit result:
\begin{proposition}
	Assume that we have an AR(1)-process with innovations $X_i$ of the form \eqref{eq:phase_jumps} with $Exp(\mu)$-distributed $X_i^+.$ Then,
	\[\phi(y)=\frac{\sum_{n\in\N}e^{\lambda^n\mu y-\eta(\lambda^n \mu)}\rho^n} {\sum_{n\in\N_0}e^{\lambda^n\mu y-\eta(\lambda^{n+1} \mu)-\psi_2(\lambda^n \mu)}\rho^n} \cdot\int_0^\infty g(y+x)\mu e^{-\mu x}dx.\]
	If $X_i^-=0$ a.s., this simplifies to
	\[\phi(y)=\frac{e^{-\mu (1 - \lambda) y}}{\gamma +
		e^{-\mu (1 - \lambda) y}}\cdot \int_0^\infty g(y+x)\mu e^{-\mu x}dx,\;\;\gamma = \frac{1}{\rho} -
	1.\]
\end{proposition}

The previous proposition indeed allows for explicit solutions to optimal stopping problems of interest. Let us, e.g., consider the case $g(x)=x^+$ with exponentially distributed innovations. Then, for $y\geq 0$,
\[\int_0^\infty g(y+x)\mu e^{-\mu x}dx=y+\frac{1}{\mu}.\]
Therefore, the optimal threshold $a^*$ is given as the unique positive solution to
\[y=\frac{e^{-\mu (1 - \lambda) y}}{\gamma +
	e^{-\mu (1 - \lambda) y}}\left(y+\frac{1}{\mu}\right),\]
i.e. 
\[a^*=\frac{LambertW((1-\lambda)/\gamma)}{(1-\lambda)\mu},\]
where $LambertW(\cdot)$ denotes the Lambert-W function.

\begin{remark}
	The corresponding result for random walks with jumps of the type \eqref{eq:phase_jumps} with exponentially distributed $X_1^+$ can be obtained in a similar manner. In this case, $\phi(y)=E\rho^{\tau_+}=c$, say, is a constant which can now be found by applying \eqref{eq:exp_innov_decomp} to the martingale $\rho^nb^{S_n},\,n\in\N_0$, where $b$ is such that $Eb^{X_1}=1/\rho.$ Then, using optional sampling,
	\[1=E\rho^{\tau_+}b^{S_{\tau_+}}=c\int_0^\infty b^x\mu e^{-\mu x}dx,\]
	i.e. 
	\[c=\frac{\mu-\log(b)}{\mu}.\]
%	We obtain
%	\begin{proposition}
%		Assume that we have a random walk with jumps $X_i$ of the form \eqref{eq:phase_jumps} with $Exp(\mu)$-distributed $X_i^+.$ Then,
This yields
		\[\phi(y)=\frac{\mu-\log(b)}{\mu} \int_0^\infty g(y+x)\mu e^{-\mu x}dx,\]
		allowing for more direct computations in the random walk examples at the beginning of this section. For example, in the situation of Subsection \ref{subsec:logistic} above for $\eta=\mu$ (for simplicity of expressions), we obtain
		\begin{align*}\phi(y)&=\frac{\mu-\log(b)}{\mu}\int_0^\infty \frac{1}{K\exp(\mu(x+y))+1}\mu e^{-\mu x}dx\\&=\frac{\mu-\log(b)}{\mu K}\exp(\mu y)\log\left(\frac{\exp(\mu y)+K}{\exp(\mu y)}\right),
		\end{align*}
		which can be used to (partly) extend the continuous time results in \cite{RePEc:tkk:dpaper:dp9}.
%	\end{proposition}
\end{remark}

\subsection{Sum-the-odds}
We now come to a structurally different problem, which interestingly also fits well into our theory: the sum-the-odds theorem introduced in \cite{MR1797879}. The optimal stopping problem is the problem of maximizing the probability of stopping on the last success of a finite sequence of independent Bernoulli trials. More precisely, let $n\in\N$ be a fixed positive integer and $X_1,X_2,\dots,X_n$ be independent random variables with values in $\{0,1\}$. The random variables are not assumed to be identically distributed. We denote the success probabilities by $p_k=P(X_k=1)=1-P(X_k=0)=1-q_k$. To give a solution in our framework, we introduce a Markov chain on 
\[\{-n,n+1,\dots, -1,0,1,\dots,n-1,n\}\]
 with transition probabilities as follows: The states $n$ and $-n$ are absorbing and 
\begin{align*}
&P_k(Y_1=k+1)=p_{k+1}=1-P_k(Y_1=-k-1)\mbox{ for }k\geq 0,\\
&P_k(Y_1=-k+1)=p_{k+1}=1-P_k(Y_1=k-1)\mbox{ for }k< 0.
\end{align*}
In other words, if we let the process start in 0, then $Y_k\in\{-k,k\}$ and $Y_k=k$ corresponds to a success at time $k$. The reward function $g$ describing the problem is now given by $g(k)=0$ for $k\leq 0$ and 
\[g(k)=P_k(Y_{k+1}<0,\dots,Y_n<0)=\prod_{l=k+1}^{n}q_l\mbox{ for }k>0.\]
Furthermore, for $k\geq 1$
\[\phi(k)=E_kg(Y_{\tau_k})=\sum_{l=k+1}^{n}g(l)P_k(Y_{\tau_k}=l)=\sum_{l=k+1}^{n}g(l)p_l\prod_{r=k+1}^{l-1}q_r=\sum_{l=k+1}^{n}p_l\prod_{\substack{r=k+1\\ r\not=l}}^{n}q_r \]
and
\[P_k(\tau_k=\infty)=g(k).\]
Therefore,
\[f(k)=\frac{g(k)-\phi(k)}{P_k(\tau_k=\infty)}=1-\sum_{l=k+1}^{n}p_l\left(\prod_{\substack{r=k+1\\ r\not=l}}^{n}q_r\right)\left(\prod_{\substack{r=k+1}}^{n}q_r\right)^{-1}=1-\sum_{l=k+1}^{n}\frac{p_l}{q_l}.\]
The assumptions of Theorem \ref{thm:optimal} are obviously fulfilled, so that the optimal threshold $\alpha^*$ is given as the smallest $k\geq 1$ such that $f(k)\geq 0$, i.e.
\[\sum_{l=k+1}^{n}\frac{p_l}{q_l}\leq 1,\]
reproducing the odds-theorem (Theorem 1) in \cite{MR1797879}. Let us just mention that some of the generalisations surveyed in \cite{dendievel2012new} may be handled similarly. 

\subsection{Connection to other approaches}\label{subsec:other_approaches}
We now describe the connection of the approach described in this paper and methods going back to \cite{NS}, \cite{NS2} for random walks (and L\'evy processes in continuous time) and the generalization to general Markov processes on the real line in \cite{CST}. Note that \cite{CST} considers the continuous time setting. The adaption to discrete time is, however, straightforward, so that we use this without giving proofs, but refer to the more detailed discussion in \cite{doi:10.1080/07474946.2016.1275314}. The starting point is a representation of the reward function $g$ in terms of the running maximum. Using the notation
\[M_n=\sup_{m\leq n}Y_m,\]
an independent random variable $T$\ with\ geometric distribution with parameter $1-\rho$, the first step is to guess a function\ $\tilde f$ such that the lower semicontinuous reward function $g$ can be represented as
\begin{equation}\label{eq:repr_g}g(y)=E_y{\tilde f}(M_T).
\end{equation}
\cite{CST}, Theorem 1, states that \eqref{eq:M1} for $\tilde f$ instead of $f$ implies that 
\[\inf\{n\geq 0:Y_n>a^*\}\]
is optimal. Indeed, our function $f$ introduced in \ref{eq:f_def} and the representing function $\tilde f$ coincide:
%Using the notation
%\[{\tau_y}=\inf\{n>0:Y_n>y\}, \]
%it holds that
\begin{align*}
E_x\rho^{\tau_y}g(Y_{\tau_y})&=E_x1_{\{\tau_y\leq T\}}E_{Y_{\tau_y}}({\tilde f}(M_T))=E_x1_{\{\tau_y\leq T\}}E_{x}({\tilde f}(M_T)|\mathcal{A}_{\tau_y})\\
&=E_x{\tilde f}(M_T)1_{\{\tau_y\leq T\}}=E_x{\tilde f}(M_T)1_{\{M_T>y\}},
\end{align*}
where we used the memoryless-property of $T$.
Therefore, for $g$ with representation of the form \eqref{eq:repr_g},
\[g(y)-\phi(y)=E_y{\tilde f}(M_T)-E_y{\tilde f}(M_T)1_{\{M_T>y\}}=f(y)P_y(M_T=y),\]
so that
\begin{equation*}
\tilde f(y)=\frac{1}{P_y(M_T=y)}(g(y)-\phi(y))=f(y).
\end{equation*}
This shows that Theorem \ref{thm:optimal} implies the results in the literature. The proofs there are, however, different in nature. The results for random walks (and L\'evy processes) are based on the Wiener-Hopf factorization and the results in the general Markov case heavily rely on representation results for excessive functions. 
Another main difference is that in this paper we give an explicit formula for the function $f$, which allows to obtain general results such as Theorem \ref{prop:RW_optimal}, whereas the existence of a representing function $\tilde{f}$ in \eqref{eq:repr_g} is not clear in general (but see the discussion in \cite{CST},\ Subsection 2.2 and -- for the L\'evy process case using Fourier techniques -- Lemma 3.1 in \cite{Su}).
%we do not need a guess of $\tilde f$ to apply the theory, but can start directly with the given functions $g,\,\phi$. 
%Furthermore, the results in Section \ref{sec:monotone_case} are not restricted to the case of a monotone $\tilde f$, where the previous papers do not provide any results.

Another related approach is described in \cite{woodroofe1994generalized}. For an underlying random walk with positive drift and no discounting, concave reward functions $g$ with a unique maximum were considered. This implies that $g$ is unbounded from below, so that we leave the setting of this paper. One main result is that in this situation, the optimal stopping time is of the form
\[\inf\{n\geq 0:Y_n>a^*\}.\]
%and $a^*$ is characterized in terms of the ascending ladder variables. 
The authors also use ascending ladder variables to obtain their result and to characterize $a^*$. In our notation, they use the fact that $\phi(y)-g(y)$ is decreasing in $y$ due to the concavity and monotonicity, establishing \eqref{eq:M}. Then,
$a^*$ is characterized analogously to our $\alpha^*$. The line of argument heavily relies on the Wiener-Hopf factorization, but a similarity to our proof of Theorem \ref{thm:optimal} can also be found.  

\section{Some remarks on problems in continuous time}\label{sec:cont_time}
The focus of this paper is on processes in discrete time as ideas can be explained without technical difficulties in this setting. Nonetheless, for continuous time processes, a similar approach can be used to tackle one-sided optimal stopping problems of the form
\begin{equation}\label{eq:reward_cont}
e^{-rt}g(Y_t),\;t\in[0,\infty),\;r\in[0,\infty),\;g:\R\rightarrow[0,\infty).
\end{equation}
 However, the notations have to be handled more carefully. One first problem is that, in general, new maxima do not only occur at discrete time points, so that we have to use the local time at the maximum instead. There are different approaches to the notion of local time for different classes of general Markov processes and semimartingales. To avoid this technical discussion and to sketch the ideas very clearly, we from now on assume  $Y$ to be a L\'evy process and use the concept of local time $L$ at the maximum as presented in \cite{kyp}. 
% To furthermore simplify notation, we in addition assume that $r=0$; note that the general case may be handled by considering $Y$ killed at an independent exponentially with parameter $r$ distributed random time $T$.\\
% Generalizations to more general Markov processes along the lines presented below are straightforward. \\
 For L\'evy processes $Y$, the ladder height process is defined by 
 \[H_s=Y_{L_s^{-1}}\mbox{ on }\{L_s^{-1}<\infty\}\mbox{ and }H_s=\Delta\mbox{ on }\{L_s^{-1}=\infty\} \]
where $L^{-1}_s$ denotes the right inverse, so that $L^{-1}$ is the ascending ladder time process. Then the process
\[(L^{-1}_s,H_s)_{s\geq 0}\]
is a -- potentially killed -- subordinator, in particular a nice Markov process and semimartingale. For all\ $y>0$, it holds that
\[\inf\{t\geq 0:Y_t\in I(y)\}=\inf\{t\geq 0:H_{L_t}\in I(y)\}=L^{-1}_{\gamma_{y}},\]
where $\gamma_y=\inf\{s:H_s\in I(y)\}$
as the stopping time is in the support of\ the measure $dL$. Assume again that the optimal stopping time for \eqref{eq:reward_cont} is of threshold-type for some $a^*$. Adapting the arguments in discrete time, we can now equivalently consider the optimal stopping problem for the process
\begin{equation}\label{eq:reward_local_time}
e^{-rL^{-1}_s}g(H_s),\;s\in[0,\infty).
\end{equation}
Here, one can hope for a monotone-type situation for \eqref{eq:reward_local_time}. However, the notion of monotone problems is more involved in the continuous time setting, see \cite{MR731723,MR533005,MR1030003}. We assume that the process in \eqref{eq:reward_local_time} has a Doob-Meyer-type decomposition of the form
\[e^{-rL^{-1}_s}g(H_s)=g(H_0)+M_s+\int_0^s\tilde{f}(H_u)dV_u,\;\;{s\geq 0},\]
where $(V_s)_{s\in[0,\infty)}$ is increasing. 
%To use this in our setting, 
To illustrate the connection to the general Markov process theory, note in the case $r=0$ (for simplicity only) the following: If $g$ is in the domain of the extended generator of $H$, cf. \cite[VII.1]{RY},
then, by definition, there exists a function
${\tilde f}$ such 
that 
%$$\int_{0}^{s}|{\tilde f}(H_u)|du <\infty$$ 
%and 
the process
$$g(H_s)-g(H_0)-\int_0^s \tilde{f}(H_u)du,\;\;{s\geq 0},$$
is a martingale for all initial states of\ $H_0$. Note that, when $g$ is in the domain of the infinitesimal generator $A$ of $H$, then $\tilde f=Ag$. Hence, $\tilde f$ is the continuous time analogue to $\phi-g$.
Coming back to the general setting, the stopping problem for reward process \eqref{eq:reward_local_time} is called monotone if 
\begin{align}\tag{$M_{cont}$}\label{eq:M_cont}
\tilde f(z)\leq 0\mbox{ implies }\tilde f(z')\leq 0\mbox{ for }z'>z>b.
\end{align}
Under the assumption corresponding to \eqref{eq:Opt}, this yields that the myopic stopping time 
\[\nu^*=\inf\{s\geq 0:H_s\in I(\alpha^*)\}, \]
with
	\[\alpha^*:=\{z>b:\tilde f(z)\leq 0\},\]
is optimal, see Appendix A in \cite{ChristensenIrle17}. This yields a result similar to \ref{thm:opt_threshold} in this continuous time setting. 

Using the ideas developed above, continuous time versions of Theorems \ref{thm:optimal} and \ref{prop:RW_optimal} could also be obtained. This, however, is technically more challenging. For not necessarily smooth reward functions, one major challenge is to find a representation of the form \eqref{eq:M_cont} to work with. In our situation, an approximation argument is however easy to carry out in order to take over the results to the continuous time setting. The following proof follows the line of argument in \cite{Bei98}:

	\begin{thm}
		Assume that $Y$ is a L\'evy process, $g$ is increasing and log-concave on $\{y:y>b\}$ and \eqref{eq:Opt} holds true. Then, there exists $\alpha^*>b$ such that			
		\[\tau^*=\inf\{t\geq 0:Y_t\geq \alpha^*\}\]
		is an optimal stopping time.
		%	Assume \eqref{eq:Opt} and for $\rho=1$ that $P_y(\tau_y=\infty)>0$.
		%	Assume that 
		%	\begin{align}\tag{M1}\label{eq:M1}
		%	%	f(y)\leq 0\mbox{ for $y\leq \alpha^*$ and 
		%	\mbox{$f$ is 
		%		%	positive\footnote{ist $f\geq 0$ ok?} and 
		%		increasing for $y\geq \alpha^*$.}
		%	\end{align}
%		\begin{enumerate}[(i)]
%			\item Assume \eqref{eq:A1}. Then
%			\[\tau^*=\inf\{n\geq 0:Y_n\geq \alpha^*\}\]
%			is an optimal stopping time for problem \eqref{eq:OSP}, i.e. $a^*=\alpha^*$.
%			\item Assume \eqref{eq:A2}. Then
%			\[\tau^*=\inf\{n\geq 0:Y_n>\alpha^*\}\]
%			is an optimal stopping time for problem \eqref{eq:OSP}.
%		\end{enumerate}
	\end{thm}

\begin{proof}
	For each $n\in\N$ we consider the discrete time process $Y^{(n)}$ given by 
	\[Y^{(n)}_{k}=Y_{2^{-n}k},\;\;k=0,1,2,\dots . \]
	With respect to the corresponding filtration $\A^{(n)}$ given by $\A^{(n)}_k=\A_{2^{-n}k}$, this process is a random walk. Therefore, the optimal stopping problem with reward process
	\[e^{-rk2^{-n}}g(Y^{(n)}_{k}),\;k=0,1,2,\dots \]
	falls within the framework of Theorem \ref{prop:RW_optimal}. As $g$ is continuous on $(b,\infty)$ and \eqref{eq:Opt} holds, we are in case \eqref{eq:A1}, so that there exist thresholds $b<\alpha_1, \alpha_2,\dots$ such that the optimal stopping sets for the discrete time processes are given by
	\[S_n=\{y:v_n(y)=g(y)\}=[\alpha_n,\infty).\]
	Here $v_n$ denotes the value function for the discrete time process. As the set of admissible stopping times is growing in $n$, so is $v_n$, and hence also the continuation sets. Therefore, we have $b<\alpha_1\leq \alpha_2\leq \dots$.

	Moreover, for each $y$, $v_n(y)$ converges to the value $v(y)$ of the continuous time problem as $n\rightarrow\infty$. Indeed, for each stopping time $\tau$ for the continuous time process, we can consider the discrete time approximation
	\[\tau_n=\inf\{k2^{-n}:\tau\leq k2^{-n}\}.\]
	Then, $\tau_n\searrow\tau$ and, as $Y$ has right continuous sample paths, $Y_{\tau_n}\rightarrow Y_\tau$, so that due to \eqref{eq:Opt} it holds that
	\[E_ye^{-r\tau_n}g(Y_{\tau_n})\rightarrow E_ye^{-r\tau}g(X_\tau).\]
	Applying this result to the optimal stopping time $\tau=\tau^*$ for the continuous process, we obtain 
	\[v(y)=E_ye^{-r\tau^*}g(Y_{\tau^*})=\lim_{n\rightarrow\infty}E_ye^{-r\tau_n}g(Y_{\tau_n})\leq \lim_{n\rightarrow\infty}v_n(y)\leq v(y),\]
	proving the convergence of the value functions. 
We now write $\alpha^*=\lim_{n\rightarrow\infty}\alpha_n.$ Then for all $y\geq \alpha^*$, 
\[v(y)=\lim_{n\rightarrow\infty}v_n(y)=g(y),\]
so that $y$ is in the stopping set for the continuous time problem. On the other hand, for $y<\alpha^*$ there exists $n\in\N$ such that
\[g(x)<v_n(x)\leq v(x). \]
Hence, $[\alpha^*,\infty)$ is the stopping set for the continuous time problem and the threshold time for $\alpha^*$ is optimal due to the general theory. 
\end{proof}

\appendix
\section{Appendix}
\subsection{An illustration for an elementary first step}\label{sec:appendix_illust}
This paper starts from the observation that for many well-known stopping problems one can show by elementary methods that optimal stopping times are of threshold-type. Here, we want to give a short illustration what we mean by the term \emph{elementary}. To show that the optimal stopping set is of the type $[a,\infty)$ or $(a,\infty)$ for some $a\in\R$ we have to show that for $y'>y$ in the range of interest
\[V(y)\leq g(y)\mbox{ implies }V(y')\leq g(y').\]
Let us illustrate this for the Shepp-Shiryaev problem, following \cite{baurdoux2013direct}. In the setting of Subsection \ref{subsec:Shepp_shi} we have
\begin{align*}
	V(y)&=\sup_{\tau}E\rho^\tau\exp(\max\{y,S_0,\dots,S_\tau\})\\
	&=E\rho^{\tau^*_y}\exp(\max\{y,S_0,\dots,S_{\tau^*_y}\}),
\end{align*}
where ${\tau^*_y}$ denotes an optimal stopping time with respect to $y\geq 0$. Let $y'>y\geq 0$, ${\tau^*}={\tau^*_{y'}}$. Then
\begin{align*}
V(y')-V(y)&\leq E\rho^{\tau^*}\exp(\max\{y',S_0,\dots,S_{\tau^*}\})-E\rho^{\tau^*}\exp(\max\{y,S_0,\dots,S_{\tau^*}\})\\
&=E\rho^{\tau^*}\big(\exp(\max\{y',S_0,\dots,S_{\tau^*}\})-\exp(\max\{y,S_0,\dots,S_{\tau^*}\})\big)\\
&\leq E\rho^{\tau^*}(e^{y'}-e^{y})\leq e^{y'}-e^{y}=g(y')-g(y).
\end{align*}
So,
\[V(y')-g(y')\leq V(y)-g(y)\mbox{ for }y'>y\geq 0\]
and 
\[V(y)\leq g(y)\mbox{ implies }V(y')\leq g(y').\]
We refer to \cite{CIN} and \cite{baurdoux2013direct} for similar arguments of this type, and to Theorem \ref{thm:optimal} of this paper for a more sophisticated and general result. 

\subsection{On property \eqref{eq:Opt} in the Shepp-Shiryaev problem}\label{sec:appendix_opt_ss_problem}
In the setting of Subsection \ref{subsec:Shepp_shi}, we now investigate the validity of condition \eqref{eq:Opt}, so we look at \[E_Q\sup_nr^ne^{Y^y_n}=E\sup_n\rho^ne^{\max\{y,S_0,\dots,S_n\}}.\]
 The main tools will be the results from \cite{sgibnev1997submultiplicative}. 
Let us just mention that the case $EX_i\geq 0$ can be treated more elementary. 

Taking\ $y=0$ w.l.o.g., we have
\[E\sup_n\rho^n\exp({\max_{i\leq n}S_i})\leq E\sup_n\max_{i\leq n}\rho^i\exp(S_i)=E\sup_{i\geq 0}\rho^i\exp(S_i)=E\exp(\sup_{i\geq 0}S_i'), \]
say, with\ $S'_i=\sum_{j=1}^iX_j',\;X_j'=X_j+\log(\rho).$
Here, 
\[E\exp(X_j')=E\exp(X_j)\rho=\mu\rho<1\]
and
\[EX_j'\leq \log E\exp(X_j')\leq \log(\mu\rho)<0.\]
These are the two conditions to apply Theorem 2 in \cite{sgibnev1997submultiplicative}, which yields
\[E\exp(\sup_{i\geq 0}S_i')<\infty.\]

To show the second property in \eqref{eq:Opt} just choose $\rho'$ with $\rho<\rho'<1,\,\rho'\mu<1$. Then with $r'=\rho'\mu<1$
\[r^n\exp(Y_n^y)=\left(\frac{\rho}{\rho'}\right)^n(r')^n\exp(Y_n^y)\]
implying
\[r^n\exp(Y_n^y)\rightarrow 0\;\;Q-\mbox{a.s.}\]
since $\frac{\rho}{\rho'}<1$ and $E_Q\sup_n(r')^n\exp(Y_n^y)<\infty$ as before.

\bibliographystyle{abbrvnat}
\bibliography{Irle_monoton_eindim}

\end{document}